\documentclass{amsart}
\usepackage{mathrsfs}
\usepackage{amsmath,amssymb}
\usepackage{hyperref}

\theoremstyle{plain}
\newtheorem{theorem}{Theorem}[section]
\newtheorem{lemma}{Lemma}[section]

\newtheorem{corollary}{Corollary}[section]

\theoremstyle{definition}

\newtheorem{example}{Example}[section]

\theoremstyle{remark}

\newtheorem{question}{Question}[section]

\begin{document}

\title[About remainders in compactifications of paratopological groups]
{About remainders in compactifications of paratopological groups}

\author{Fucai Lin}
\address{Fucai Lin(corresponding author): Department of Mathematics and Information Science,
Zhangzhou Normal University, Zhangzhou 363000, P. R. China}
\address{Fucai Lin: Department of Mathematics, Sichuan University, Chengdou, 610064, P.R.China}%
\email{linfucai2008@yahoo.com.cn}

\author{Shou Lin}
\address{Shou Lin: Institute of Mathematics, Ningde Teachers' College, Fujian\, 352100, P. R. China}%
\email{linshou@public.ndptt.fj.cn}

\thanks{Supported by the NSFC (No. 10971185) and the Educational Department of Fujian Province (No. JA09166) of China.}


\begin{abstract}
In this paper, we prove a dichotomy theorem for remainders in
compactifications of paratopological groups: every remainder of a
paratopological group $G$ is either Lindel\"{o}f and meager or
Baire. Moreover, we give a negative answer for a question posed by
D. Basile and A. Bella in \cite{B1}, and some questions about
remainders of paratopological groups are posed in the paper.
\end{abstract}

\maketitle

\section{Introduction}
By a remainder of a space $X$ we understand the subspace
$bX\setminus X$ of a Hausdorff compactification $bX$ of $X$.
Remainders in compactifications of topological spaces have been
studied by some topologists in the last few years. A famous
classical result in this study is the following theorem of M.
Henriksen and J. Isbell \cite{H1958}:

\medskip
{\bf (M. Henriksen and J. Isbell)} A space $X$ is of countable type
if and only if the remiander in any (in some) compactification of
$X$ is Lindel\"{o}f.

Since topological groups are much more sensitive to the properties
of their remainders than topological spaces in general, topologists
are mainly interesting in the remainders of topological groups or
paratopological groups. For instance, Arhangel'ski\v{\i} has
recently proved the following two dichotomy theorems about
remainders in compactifications of topological groups:

\begin{theorem}\label{l0}\cite{A3}
If $G$ is a topological group, and some remainder of $G$ is not
pseudocompact, then every remainder of $G$ is Lindel\"{o}f.
\end{theorem}

\begin{theorem}\label{l2}\cite{A2}
Suppose that $G$ is a non-locally compact topological group. Then
either every remainder of $G$ has the Baire property, or every
remainder of $G$ is $\sigma$-compact.
\end{theorem}

Moreover, D. Basile and A. Bella have just shown a dichotomy theorem for homogeneous spaces:

\begin{theorem}\cite{B1}\label{l6}
The remainder of a homogeneous space is either Baire or meager and realcompact.
\end{theorem}

D. Basile and A. Bella posed the following question.

\begin{question}\cite{B1}\label{q0}
Let $X$ be a homogeneous space and let $bX$ be a compactification of
$X$. Is it true that the remainder $bX\setminus X$ is either
pseudocompact or realcompact and meager?
\end{question}

In \cite{B1}, D. Basile and A. Bella has shown that none of the
Arhangel'ski\v{\i}'s dichotomy theorems can be generalized to the
case of homogeneous spaces. In \cite{A1}, Arhangel'ski\v{\i} has
given an example to show that the Theorem~\ref{l0} can not be
generalized to the case of paratopological groups. Naturally, the
following two questions arise:

\begin{question}\label{q5}
How about the dichotomy theorem of the remainders in compactifications of paratopological groups?
\end{question}

\begin{question}\label{q6}
Can the dichotomy theorem~\ref{l2} be generalized to the case of paratopological groups?
\end{question}

In this paper, we show that, for a paratopological group $G$, every
remainder of $G$ is either Lindel\"{o}f and meager or Baire, which
give an answer for Question~\ref{q5}. Also, we give a partial answer
for the Question~\ref{q6}. Finally, we give a negative answer for
the Question~\ref{q0}. \maketitle

\section{Preliminaries}
Recall that a {\it topological group} $G$ is a group $G$ with a
topology such that the product map of $G \times G$ into $G$
associating $xy$ with arbitrary $(x, y)\in G\times G$ is jointly
continuous and the inverse map of $G$ onto itself associating
$x^{-1}$ with arbitrary $x\in G$ is continuous. A {\it
paratopological group} $G$ is a group $G$ with a topology such that
the product map of $G \times G$ into $G$ is jointly continuous. A
{\it semitopological group} $G$ is a group $G$ with a topology such
that the product map of $G\times G$ into $G$ is separately
continuous. A {\it quasitopological group} $G$ is a group $G$ with a
topology such that it is a semitopological group and the inverse map
of $G$ onto itself is continuous.

Recalled that a space is {\it Baire} if the intersection of a
sequence of open and dense subsets is dense. Moreover, a space is
called {\it meager} if it can be represented as the union of a
sequence of nowhere dense subsets.

Let us call a map $f$ of a space $X$ into a space $Y$ {\it
$k$-gentle} \cite{A1} if for every compact subset $F$ of $X$ the
image $f(F)$ is also compact. A semitopological group $G$ will be
called {\it $k$-gentle} \cite{A1} if the inverse map ($x\mapsto
x^{-1}, \forall x\in G$) is $k$-gentle.

A family $\mathcal{A}$ of open subsets of a space $X$ is called {\it
a base of $X$ at a set $A$} if $A=\cap\mathcal{A}$ and for any
neighborhood $U$ of $A$, there is a $V\in\mathcal{A}$ such that
$A\subset V\subset U$. If $\mathcal{A}$ is countable, then we say
that $A$ has countable character in $X$. A space $X$ is of {\it
countable type} \cite{E} if every compact subspace $F$ of $X$ is
contained in a compact subspace $K\subset X$ with a countable base
of open neighborhoods in $X$.

Throughout this paper, all spaces are assumed to be Tychonoff.
Denote positively natural number by $\mathbb{N}$. We refer the
reader to \cite{A2008, E} for notations and terminology not
explicitly given here.
\bigskip

\section{Remainders of paratopological groups}
Firstly, we give a lemma.

\begin{lemma}\label{l3}\cite{A1}
Let $G$ be a paratopological group. If there exists a non-empty
compact subset of $G$ of countable character in $G$, then $G$ is of
countable type.
\end{lemma}

Now, we give a dichotomy theorem of the remainders in
compactifications of paratopological groups.

\begin{theorem}\label{t0}
Let $G$ be a non-locally compact paratopological group. Then either
every remainder of $G$ has the Baire property, or every remainder of
$G$ is meager and Lindel\"{o}f.
\end{theorem}

\begin{proof}
Suppose that $bG$ is a compactification of $G$ such that the
remainder $Y=bG\setminus G$ does not have the Baire property. Next,
we shall prove that $Y$ is Lindel\"{o}f and meager.

Since $Y$ does not have the Baire property, there exists a countable
family $\{U_{n}: n\in \mathbb{N}\}$ of open subsets of $Y$ such that
$\cap\{U_{n}: n\in \mathbb{N}\}$ is not dense in $Y$. Because $G$ is
nowhere locally compact, $Y$ is dense in $bG$. For each $n\in
\mathbb{N}$, there exists an open subset $V_{n}$ of $bG$ such that
$U_{n}=V_{n}\cap Y$. Let $\gamma =\{V_{n}: n\in \mathbb{N}\}$.
Therefore, we can find a non-empty open subset $U$ of $bG$ such that
$(\cap\gamma )\cap (U\cap Y)=\emptyset$. It follows that the
subspace $Z=(\cap\gamma )\cap (U\cap G)=(\cap\gamma )\cap U$ is
\v{C}ech-complete in $U\cap G$ of $G$ by \cite[Theorem 3.9.6]{E}. It
is known that every \v{C}ech-complete space is of countable type.
Since $Z$ is \v{C}ech-complete, there exists a non-empty compact
subset $F$ of $Z$ of countable character in $Z$. Because $Z$ is
dense in the open subspace $U\cap G$ of $G$, $F$ is of countable
character in $U\cap G$ \cite{A2008}. Because $U\cap G$ is open in
$G$, $F$ is of countable character in $G$. Obviously, $F$ is compact
in $G$. Therefore, $G$ is of countable type by Lemma~\ref{l3}.
Therefore, $Y$ is Lindel\"{o}f by M. Henriksen and J. Isbell
theorem. Moreover, $Y$ is meager by Theorem~\ref{l6}. This complete
the proof.
\end{proof}

{\bf Remark}\, Observe that a remainder $Y$ of a non-locally compact
paratopological group $G$ cannot have the Baire property, be
Lindel\"{o}f and meager at the same time. Indeed, it is easy to see
that the failure of the Baire property is equivalent to the
existence of some non-empty open meager subset. Thus we have the
following two corollaries.

\begin{corollary}
Let $X$ be a neither Baire nor meager space. Then $X$ cannot be a
remainder in compactifications of any paratopological group.
\end{corollary}

\begin{corollary}
Let $X$ be a neither Baire nor Lindel\"{o}f space. Then $X$ cannot
be a remainder in compactifications of any paratopological group.
\end{corollary}

{\bf Remark}\, D. Basile and A. Bella has shown that there exists a
homogeneous space such that the remainder of some compactification
is neither Baire nor Lindel\"{o}f, see \cite[Example 3.3]{B1}. Hence
Theorem~\ref{t0} can not be generalized to the case of homogeneous
spaces. However, we have the following question.

\begin{question}\label{q1}
Let $X$ be a non-locally compact semitopological group or
quasitopological group, and let $bX$ be a compactification of $X$.
Is it true that the remainder $bX\setminus X$ has the Baire property
or is Lindel\"{o}f and meager?
\end{question}

Next, we obtain two corollaries from Theorem~\ref{t0}. Firstly, we
show that the Arhangel'ski\v{\i}'s dichotomy Theorems~\ref{l2} can
be generalized to the case of $k$-gentle paratopological groups,
which give a partial answer for Question~\ref{q6}.

\begin{lemma}\label{l1}\cite{A1}
Let $G$ be a $k$-gentle paratopological group such that some
remainder of $G$ is Lindel\"{o}f. Then $G$ is a topological group.
\end{lemma}

\begin{corollary}\label{c0}
Let $G$ be a non-locally compact $k$-gentle paratopological group.
Then either every remainder of $G$ has the Baire property, or every
remainder of $G$ is $\sigma$-compact.
\end{corollary}

\begin{proof}
Suppose that $bG$ is a compactification of $G$, and put
$Y=bG\setminus G$. By Theorem~\ref{t0}, $Y$ has the Baire property,
or is meager and Lindel\"{o}f. Suppose that $Y$ does not have the
Baire property. Then $Y$ is Lindel\"{o}f, and hence $G$ is a
topological group by Lemma~\ref{l1}. Then $Y$ is $\sigma$-compact by
Theorem~\ref{l2}.
\end{proof}

It follows from \cite{A3} that a remainder in some compactification
of a topological group is metacompact iff it is Lindel\"{o}f iff it
is realcompact. Therefore, we have the following question.

\begin{question}\label{q4}
Assume that $G$ is a non-locally compact paratopological group, and
put $Y=bG\setminus G$. Are the following conditions equivalent?
\begin{enumerate}
\item $Y$ is metacompact;

\item $Y$ is Lindel\"{o}f;

\item $Y$ is realcompact.
 \end{enumerate}
\end{question}

A space $X$ is called {\it metacompact} if each open covering of $X$
can be refined by a point-finite open covering. A space $X$ is
called {\it ccc} if every disjoint family of open subsets of $X$ is
countable.

\begin{lemma}\cite{BD}\label{l5}
Every point-finite open collection in a ccc Baire space is countable.
\end{lemma}

The next corollary gives a partial answer for the Question~\ref{q4}.

\begin{corollary}
Assume that $G$ is a non-locally compact paratopological group, and
put $Y=bG\setminus G$. If $Y$ is metacompact and ccc, then $Y$ is
Lindel\"{o}f.
\end{corollary}

\begin{proof}
By Theorem~\ref{t0}, $Y$ has the Baire property, or is meager and
Lindel\"{o}f. Suppose that $Y$ has the Baire property. Then $Y$ is
Lindel\"{o}f by Lemma~\ref{l5}. Hence $Y$ is Lindel\"{o}f.
\end{proof}

Now, we shall give a negative answer for Question~\ref{q0} by
Example~\ref{e0}.

\begin{example}\label{e0}
There exists a paratopological group $X$ such that some
compactification $bX$ of $X$ has a remainder which is neither
pseudocompact nor meager.
\end{example}

\begin{proof}
Let $Z=X\cup Y$ be the two-arrows space of P. S. Alexandroff and P.
S. Urysohn \cite[Exercise 3.10. C]{E}, where $X=\{(x, 0): 0<x\leq
1\}$ and $Y=\{(x, 1): 0\leq x< 1\}$. The space $X$ is the arrow
space which is homeomorphic to the Sorgenfrey line, see
\cite[Example 1. 2. 2]{E}. $Z$ is a Hausdorff compactification of
Sorgenfrey line $X$, and its remainder $Y$ is still a copy of
Sorgenfrey line. Moreover, there exists a natural structure of an
Abelian group on $Y$ such that the multiplication $(u, v)\mapsto
u\cdot v$ is continuous, that is, the space $Y$ admits a structure
of a paratopological group. For example, if $u=(x, 1)$ and $v=(y,
1)$ are two points in $Y$, then $u\cdot v=(x+y, 1)$ if $x+y<1$, and
$u\cdot v=(x+y-1, 1)$ if $x+y\geq 1$. However, Sorgenfrey line is
non-pseudocompact; otherwise, Sorgenfrey line is a compact space
since it is a Lindel\"{o}f space, which is a contradiction.
Moreover, since $X$ has the Baire property \cite{AI}, $X$ is non-meager.
Therefore, $Y$ is neither pseudocompact nor meager.
\end{proof}

{\bf Remark}\, It follows from Example~\ref{e0} that, in
Question~\ref{q0}, the answer is also negative if we replace the
``homogeneous space'' by ``paratopological group''.

\end{document}